\documentclass[11pt]{article}
\usepackage{amsmath,amssymb,amsthm,hyperref,theoremref,lipsum,enumitem,lmodern,lineno}

\providecommand{\keywords}[1]{\textbf{\textbf{Keywords}:} #1}
\providecommand{\subjclass}[1]{\textbf{\textbf{Mathematics subject classification(2010)}:} #1}
\newcounter{casenum}

\newcommand\mycom[2]{\genfrac{}{}{0pt}{}{#1}{#2}}
\def\re{{\Re e\,}}

\def\f{{\mathfrak{f}}}
\theoremstyle{plain}
\newtheorem*{thm*}{Theorem}
\newtheorem{thm}{Theorem}

\newtheorem{lem}{Lemma}

\theoremstyle{definition}

\numberwithin{equation}{section}

\title{A note on the extended Bruinier-Kohnen conjecture}
 \author{Mohammed Amin Amri \and M'hammed Ziane }
 \newcommand{\Addresses}{{
  \bigskip
  \footnotesize

  Mohammed Amin~Amri, \textsc{ACSA Laboratory, Department of Mathematics,Faculty of Sciences, Mohammed First University,Oujda, Morocco.}\par\nopagebreak
  \textit{E-mail address}, Mohammed Amin~Amri: \texttt{amri.amine.mohammed@gmail.com}
  
  M'hammed~Ziane, \textsc{ACSA Laboratory, Department of Mathematics,Faculty of Sciences, Mohammed First University,Oujda, Morocco.}\par\nopagebreak
  \textit{E-mail address}, M'hammed~Ziane: \texttt{ziane12001@yahoo.fr}

}}
 
\begin{document}

\maketitle
\begin{abstract}
Let $f$ be a cusp form of half-integral weight $k+1/2$, whose Fourier coefficients $a(n)$ not necessarily real. We prove an extension of the Bruinier-Kohnen conjecture on the equidistribution of the signs of $a(n)$ for the families $\{a(tp^{2\nu})\}_{p,\text{prime}}$, where $\nu$ and $t$ be fixed odd positive integer and square-free integer respectively.
\end{abstract}
\\

\section{Introduction and statement of result}
Over the recent years, the problem of sign changes of Fourier coefficients of various kind of automorphic forms has been studied by several authors \cite{Meher2018,Pr,Das,Meher2015,Amri}. The case of modular forms of half-integral weight was first considered by Bruinier and Kohnen in \cite{bruinier}, there, they proved sign change results for some subfamilies of Fourier coefficients of half-integral weight modular forms which are accessible via the Shimura-lift. Moreover, they conjectured that the signs of the sequence of Fourier coefficients of a cusp form which lies in the Kohnen's plus space are equidistributed. Through this problem in \cite{IW,IW2,Arias} Arias, Inam and Wiese proved the Bruinier-Kohnen conjecture in the special case when the Fourier coefficients of a Hecke eigenforms of half-integral weight are indexed by $tn^2$ with $t$ a fixed square-free number and $n\in\mathbb{N}$. In \cite{Amri2}, the results of \cite{IW,Arias} were generalized to Hecke eigenforms with not necessarily real Fourier coefficients, based on this and empirical evidence the first author generalized  the Bruinier-Kohnen conjecture to cusp forms in Kohnen's plus space with not necessarily real Fourier coefficients $a(n)$. More precisely, he conjectured that  
$$ 
\lim_{x\to\infty}\dfrac{\#\{n\le x\; :\; \re\{a(n)e^{-i\phi}\}\gtrless 0\}}{\#\{n\le x \; :\; \re\{a(n)e^{-i\phi}\} \neq 0\}}=\frac{1}{2},
$$
for each $\phi$ belonging to $[0,\pi)$. 

In the present note we address the above conjecture. Indeed, we shall prove that it holds for the families $\{a(tp^{2\nu})\}_{p\in\mathbb{P}}$ where $\nu$ and $t$ be a fixed odd positive integer and a square-free integer. In order to state our result we need to introduce some notations. Let $\mathbb{P}$ be the set of all prime numbers, and let $k,N$ be natural numbers, assume from now on that $k\ge2$ and $N$ be an odd and square-free integer. Let $\chi$ be a Dirichlet character modulo $4N$, write $S_{k+1/2}(4N,\chi)$ for the space of cusp forms of weight $k+1/2$ for the congruence subgroup $\Gamma_0(4N)$ with character $\chi$. In this set-up we have.

\begin{thm*}\thlabel{thm:1}
Let $\f\in S_{k+1/2}(4N,\chi)$ be a cuspidal Hecke eigenform, and let the Fourier expansion of $\f$ at $\infty$ be
$$
\f(z)=\sum_{n\ge1}a(n)q^n\quad q:=e^{2\pi i z}.
$$
Let $\nu$ be a positive odd integer, and $t$ be a square-free integer such that $a(t)\ne 0$. Let $\phi\in[0,\pi)$, define the set of primes
$$P_{> 0}(\phi,\nu):=\{p\in\mathbb{P}\; :\; \re(a(tp^{2\nu})e^{-i\phi})> 0\},$$
and similarly $P_{< 0}(\phi,\nu), P_{\neq 0}(\phi,\nu)$. Then the sequence $\{a(tp^{2\nu})\}_{p\in\mathbb{P}}$ is oscillatory (in the sense of \cite{KKP}), moreover the sets $P_{< 0}(\phi,\nu)$ and $P_{> 0}(\phi,\nu)$ have equal positive natural density, that is, both are precisely half of the natural density of the set $P_{\ne 0}(\phi,\nu)$.
\end{thm*}

The above theorem improves the result of the first author in \cite[Theorem 5, non-CM-case]{Amri2}. The proof which will be given in the next section follows broadly the same lines of the proof of \cite[Theorem 5, non-CM-case]{Amri2} and \cite[Theorem 1.4]{Amri3}. The essential ingredients are the Shimura-lift \cite{Shi}, and the Sato-Tate conjecture \cite[Theorem B]{ST}.

\section{Proof of the result}
We start by recalling some basic properties of the Shimura-lift and introduce some notations. The Shimura correspondence \cite{Shi,Niw} maps $\f$ to a Hecke eigenform $f_t$ of weight $2k$ for the group $\Gamma_0(2N)$ with character $\chi^2$. Let the Fourier expansion of $f_t$ be given by 
$$
f_t(z)=\sum_{n\ge 1} A_t(n)q^n.
$$
According to \cite{Shi}, the $n$-th Fourier coefficient of $f_t$ is given by 
\begin{equation}
A_t(n)=\sum_{d|n}\chi_{t,N}(d)d^{k-1}a\left(\frac{n^2}{d^2}t\right),\label{eq1}
\end{equation}
where $\chi_{t,N}$ denotes the character $\chi_{t,N}(d):=\chi(d)\left(\frac{(-1)^{k}N^{2}t}{d}\right)$, we let $\chi_0(d):=\left(\frac{(-1)^{k}N^{2}t}{d}\right)$. Since $\f$ is a Hecke eigenform, then so is the Shimura lift. Indeed, we have $f_t=a(t)f$ where $f$ is a normalized Hecke eigenform, write 
$$
f(z)=\sum_{n\ge 1}\lambda(n)n^{k-1/2}q^n,
$$
for its Fourier expansion at $\infty$. We shall assume that $a(t)=1$, in the general case we may apply the proof to $\frac{\f}{a(t)}$. Since $2N$ is square-free, it follows that $f$ is a Hecke eigenform without complex multiplication.

Let $\zeta$ be a root of unity belonging to $\mathrm{Im}(\chi)$, and let $p$ be a prime number such that $\chi(p)=\zeta$, it is clear that $\frac{\lambda(p)}{\zeta}$ is real. By Deligne's bound we have $\left |\frac{\lambda(p)}{\zeta}\right|\le 2$, thus we may write
\begin{equation}\label{eq:1}
\lambda(p)=2\zeta  \cos(\theta_p),
\end{equation}
for a uniquely defined angel $\theta_p\in [0,\pi]$. At this point we state the following theorem which will be crucial for our purpose. 
\begin{thm}(Barnet-Lamb, Geraghty, Harris, Taylor)\label{thmST}
Assume the set-up above. The sequence $\{\theta_p\}_p$ is equidistributed in $[0,\pi]$ as $p$ varies over primes satisfying $\chi(p)=\zeta,$ with respect to the Sato-Tate measure $\mu_{\text{ST}}:=\frac{2}{\pi}\sin^2\theta d\theta$. In particular, for any sub-interval $I\subset [0,\pi]$ we have 
$$\lim_{x\to\infty}\dfrac{\#\{p\le x\;:\; \chi(p)=\zeta,\;\theta_p\in I\}}{\#\{p\le x\;:\; \chi(p)=\zeta\}}=\frac{2}{\pi}\int_{I}\sin^2\theta d\theta.$$
\end{thm}
We shall need the following technical lemmas. 
\begin{lem}\thlabel{lem1}
Let $\zeta$ be a root of unity belonging to $\mathrm{Im}(\chi)$, then the set 
$$
\{p\in\mathbb{P} : \chi(p)=\zeta\},
$$
have natural density equal to $\frac{1}{r_\chi}$, where $r_\chi$ denotes the order of the character $\chi$.
\end{lem}
\begin{proof}
To establish the above lemma, it suffices to see
$$\{p\in\mathbb{P}:\chi(p)=\zeta\}=\coprod_{\mycom{a\in(\mathbb{Z}/4N\mathbb{Z})^*}{\chi(a)=\zeta}}\{p\in\mathbb{P}\;:\; p\equiv a\pmod{4N}\}.$$
Note that the number of the sets in the union is $\#\mathrm{ker}(\chi)=\frac{\phi(4N)}{r_{\chi}},$ hence by Dirichlet's theorem on arithmetic progressions we have
$$\lim_{x\to \infty} \dfrac{\#\{p\le x : \chi(p)=\zeta\}}{\pi(x)}=\frac{\#\mathrm{ker}(\chi)}{\phi(4N)}=\frac{1}{r_{\chi}},$$
where here and subsequently $\pi(x)$ denotes the prime counting function.
\end{proof}
\begin{lem}\thlabel{lem:1}
Assume the set-up above. Let $\nu$ be an odd positive integer. Define the set of primes $$\mathbb{P}_{>0}(\zeta,\nu):=\left\{p\in\mathbb{P}\;:\; \chi(p)=\zeta,\;\frac{a(tp^{2\nu})}{\zeta^\nu}>0\right\},$$ and similarly $\mathbb{P}_{\ge 0}(\zeta,\nu)$, $\mathbb{P}_{<0}(\zeta,\nu)$ and $\mathbb{P}_{\le 0}(\zeta,\nu)$. Then the sets 
$$
\mathbb{P}_{>0}(\zeta,\nu), \mathbb{P}_{\ge 0}(\zeta,\nu), \mathbb{P}_{<0}(\zeta,\nu),\mathbb{P}_{\le 0}(\zeta,\nu)
$$ 
have a natural density equal to $\frac{1}{2r_{\chi}}$, where $r_\chi$ denotes the order of the character $\chi$.
\end{lem}
\begin{proof}
Denote by $\pi_{>0}(x,\zeta):=\#\{p\le x : p\in\mathbb{P}_{>0}(\zeta,\nu)\}$ and similarly $\pi_{<0}(x,\zeta)$, $\pi_{\le0}(x,\zeta)$, $\pi_{> 0}(x,\zeta)$, $\pi_{\ge 0}(x,\zeta)$. Applying the M\"obius inversion formula to \eqref{eq1}, we derive that
$$
a(tn^2)=\sum_{d |n} \mu(d)\chi_{t,N}(d)d^{k-1} A_t\left(\frac{n}{d}\right).
$$
Set $n=p^\nu$ with $\chi(p)=\zeta$ and normalizing by $p^{\nu(k-1/2)}\zeta^\nu$, the above equality specialises to 
\begin{equation}\label{eq2}
\dfrac{a(tp^{2\nu})}{p^{\nu(k-1/2)}\zeta^\nu}=\dfrac{\lambda(p^\nu)}{\zeta^\nu}-\frac{\chi_{0}(p)}{\sqrt{p}}\dfrac{\lambda(p^{\nu-1})}{\zeta^{\nu-1}}.
\end{equation}
Since $f$ is a Hecke eigenform its $p^\nu$-th Fourier coefficient with $\chi(p)=\zeta$ is expressible (see \cite[Lemma 1]{Amri2}) by the following trigonometric identity
\begin{equation}\label{eq:2}
\lambda(p^\nu)=\dfrac{\sin((\nu+1)\theta_p)}{\sin\theta_p}\zeta^\nu,
\end{equation}
for $\theta_p\in (0,\pi)$, and in the limiting cases when $\theta_p=0$ and $\theta_p=\pi$ respectively we have $\lambda(p^{\nu})=(\nu+1)\zeta^\nu$ and $\lambda(p^{\nu})=(-1)^\nu(\nu+1)\zeta^\nu$, which can happen for at most finitely many primes $p$ only (see \cite[Remark 2]{Kohnen} ). Thus we may assume and do that $\theta_p\in (0,\pi)$. Altogether from \eqref{eq2} and \eqref{eq:2} we have
\begin{equation}\label{eq::2}
\dfrac{a(tp^{2\nu})}{\zeta^\nu}>0 \Longleftrightarrow \sin((\nu+1)\theta_p)>\frac{\chi_0(p)}{\sqrt{p}}\sin(\nu\theta_p).
\end{equation}
Let $\epsilon>0$ (small enough). Since for all $p>\frac{1}{\epsilon^2}$ one has $\left|\frac{\chi_0(p)}{\sqrt{p}}\sin(\nu\theta_p)\right|<\epsilon$, then from \eqref{eq::2} we have the inclusion of sets 
\begin{equation}\label{eq3}
\left\{p\leq x : \chi(p)=\zeta, \sin((\nu+1)\theta_p)>\epsilon\right\}\subset\left\{p\leq\frac{1}{\epsilon^2} : \chi(p)=\zeta\right\}\cup\{p\leq x : p\in\mathbb{P}_{>0}(\zeta,\nu)\}.
\end{equation}
On the other hand we have
$$
\sin((\nu+1)\theta_p)>\epsilon\Longleftrightarrow \theta_p\in I_{\epsilon}:=\bigcup_{j=1}^{\frac{\nu+1}{2}}\left(\frac{(2j-2)\pi+\arcsin(\epsilon)}{\nu+1},\frac{(2j-1)\pi-\arcsin(\epsilon)}{\nu+1}\right),
$$
Thus \eqref{eq3} is equivalent to 
$$
\left\{p\leq x : \chi(p)=\zeta, \theta_p\in I_\epsilon\right\}\subset\left\{p\leq\frac{1}{\epsilon^2} : \chi(p)=\zeta\right\}\cup\{p\leq x : p\in\mathbb{P}_{>0}(\zeta,\nu)\}.
$$
Consequently we have  
$$
\pi_{>0}(x,\zeta)+\pi_\zeta\left(\frac{1}{\epsilon^2}\right)\ge \#\left\{p\leq x : \chi(p)=\zeta, \theta_p\in I_\epsilon\right\},
$$
where, $\pi_\zeta(x):=\#\{p\le x : \chi(p)=\zeta\}$. Now divide the above inequality by $\pi_{\zeta}(x)$, we obtain
\begin{equation}\label{eq4}
\dfrac{\pi_{>0}(x,\zeta)}{\pi_\zeta(x)}+\dfrac{\pi_\zeta\left(\frac{1}{\epsilon^2}\right)}{\pi_\zeta(x)} \ge \dfrac{\#\left\{p\leq x : \chi(p)=\zeta, \theta_p\in I_\epsilon\right\}}{\pi_\zeta(x)}.
\end{equation}
Since $\pi_\zeta\left(\frac{1}{\epsilon^2}\right)$ is finite the term $\frac{\pi_\zeta\left(\frac{1}{\epsilon^2}\right)}{\pi_\zeta(x)}$ tends to zero as $x\to\infty$. By Theorem \ref{thmST} we have 
$$\lim\limits_{x\to\infty}\frac{\#\left\{p\leq x : \chi(p)=\zeta, \theta_p\in I_\epsilon\right\}}{\pi_\zeta(x)}=\mu_{ST}(I_\epsilon).$$
Hence a passage to the limit in \eqref{eq4} yields
$$
\liminf_{x\to\infty}\dfrac{\pi_{>0}(\zeta,x)}{\pi_{\zeta}(x)}\ge\mu_{ST}(I_{\epsilon}).
$$
Letting $\epsilon$ goes to zero in the above inequality we find $\liminf\limits_{x\to\infty}\dfrac{\pi_{>0}(\zeta,x)}{\pi_{\zeta}(x)}\ge\mu_{ST}(I)$, where $I:=\bigcup_{j=1}^{\frac{\nu+1}{2}}\left(\frac{(2j-2)\pi}{\nu+1},\frac{(2j-1)\pi}{\nu+1}\right)$. Now from \cite{Meher2017} we have $\mu_{ST}(I)=\frac{1}{2}$, therefore 
$$
\liminf_{x\to\infty}\dfrac{\pi_{>0}(\zeta,x)}{\pi_{\zeta}(x)}\ge\frac{1}{2}.
$$
A similar reasoning yields $\liminf\limits_{x\to\infty}\frac{\pi_{\le 0}(\zeta,x)}{\pi_{\zeta}(x)}\ge\frac{1}{2}$, in combination with $\pi_{\le 0}(x,\zeta)=\pi_\zeta(x)-\pi_{>0}(x,\zeta)$, one sees $\limsup\limits_{x\to\infty}\frac{\pi_{> 0}(\zeta,x)}{\pi_{\zeta}(x)}\le\frac{1}{2}$, whence $\lim\limits_{x\to\infty}\frac{\pi_{> 0}(\zeta,x)}{\pi_{\zeta}(x)}$ exists and equal to $\frac{1}{2}$. Now from Lemma \ref{lem1}, we conclude
$$
\lim_{x\to\infty}\frac{\pi_{> 0}(\zeta,x)}{\pi(x)}=\frac{1}{2r_{\chi}}.
$$
Similar arguments apply to the other cases.
\end{proof}

Now we proceed to prove our main theorem. Fix $\phi\in [0,\pi)$. Let $\zeta$ be a root of unity belonging to $\mathrm{Im}(\chi)$. If $p$ a prime  number satisfying $\chi(p)=\zeta$, from Lemma \ref{lem:1} we know that $\frac{a(tp^{2\nu})}{\zeta^\nu}$ is real, thus we my write
\begin{equation}\label{eq5}
\re(a(tp^{2\nu})e^{-i\phi})=\dfrac{a(tp^{2\nu})}{\zeta^\nu}\re(\zeta^\nu e^{-i\phi})\quad\text{with}\quad\chi(p)=\zeta.
\end{equation}
If $\mathrm{arg}(\zeta^\nu)\equiv\phi\pm\frac{\pi}{2}\pmod{2\pi}$, the sequence $\{\re(a(tp^{2\nu}) e^{-i\phi})\}_{p,\chi(p)=\zeta}$ is trivial. Assume $\mathrm{arg}(\zeta^\nu)\not\equiv\phi\pm\frac{\pi}{2}\pmod{2\pi}$, in this case the sequence $\{\re(a(tp^{2\nu}) e^{-i\phi})\}_{p,\chi(p)=\zeta}$ is not trivial. Without restriction of generality we can assume $\re(\zeta^\nu e^{-i\phi})>0$. Thus for any fixed $\epsilon>0$ (but small) from \eqref{eq::2} we have
\begin{equation}\label{eq6}
 \left\{p\in\mathbb{P} : p> \frac{1}{\epsilon^2},\;\chi(p)=\zeta,\;\theta_p\in I_{\epsilon}\right\}\subset \{p\in\mathbb{P} : \chi(p)=\zeta,\; \re(a(tp^{2\nu})e^{-i\phi})>0\},
\end{equation}
and
\begin{equation}\label{eq7}
\left \{p\in\mathbb{P} : p> \frac{1}{\epsilon^2},\;\chi(p)=\zeta,\;\theta_p\in I'_{\epsilon}\right\}\subset  \{p\in\mathbb{P} : \chi(p)=\zeta,\; \re(a(tp^{2\nu})e^{-i\phi})<0\},
\end{equation}
where 
$$
I_{\epsilon}:=\bigcup_{j=1}^{\frac{\nu+1}{2}}\left(\frac{(2j-2)\pi+\arcsin(\epsilon)}{\nu+1},\frac{(2j-1)\pi-\arcsin(\epsilon)}{\nu+1}\right)\subset [0,\pi],
$$
and
$$
I'_{\epsilon}:=\bigcup_{j=1}^{\frac{\nu+1}{2}}\left(\frac{(2j-1)\pi+\arcsin(\epsilon)}{\nu+1},\frac{2j\pi-\arcsin(\epsilon)}{\nu+1}\right)\subset [0,\pi].
$$
From Theorem \ref{thmST} we know that the sequence $\{\theta_p\}_{p,\chi(p)=\zeta}$ is equidistributed in $[0,\pi]$ with respect to the Sato-Tate measure $\mu_{ST}$, it follows that there are infinitely many primes $p$ satisfying $\chi(p)=\zeta$ such that $\theta_p\in I'_{\epsilon}$ and infinitely many primes $p$ satisfying $\chi(p)=\zeta$ such that $\theta_p\in I_{\epsilon}$, thus the sets in \eqref{eq6} and \eqref{eq7} are infinite. Hence for all root of unity $\zeta$ belonging to $\mathrm{Im}(\chi)$ the sequence $\{a(tp^{2\nu})\}_{p,\chi(p)=\zeta}$ is oscillatory, then the sequence $\{a(tp^{2\nu})\}_{p\in\mathbb{P}}$ so is.

What is left is to calculate the natural densities $\delta(P_{>0}(\phi,\nu))$ and $\delta(P_{<0}(\phi,\nu))$ of the sets $P_{>0}(\phi,\nu)$ and $P_{<0}(\phi,\nu)$. From \eqref{eq5} one observes that 
$$P_{>0}(\phi,\nu)=\coprod_{\mycom{\zeta,\text{root of unity}}{\zeta\in\text{Im}(\chi) ,\re(\zeta^\nu e^{-i\phi})>0}}\mathbb{P}_{>0}(\zeta,\nu)\bigsqcup \coprod_{\mycom{\zeta,\text{root of unity}}{ \zeta\in\text{Im}(\chi) ,\re(\zeta^\nu e^{-i\phi})<0}}\mathbb{P}_{<0}(\zeta,\nu),$$
and
$$P_{<0}(\phi,\nu)=\coprod_{\mycom{\zeta,\text{root of unity}}{\zeta\in\text{Im}(\chi) ,\re(\zeta^\nu e^{-i\phi})<0}}\mathbb{P}_{>0}(\zeta,\nu)\bigsqcup \coprod_{\mycom{\zeta,\text{root of unity}}{ \zeta\in\text{Im}(\chi) ,\re(\zeta^\nu e^{-i\phi})>0}}\mathbb{P}_{<0}(\zeta,\nu),$$
up to finitely many primes. Then the above formula combined with Lemma \ref{lem:1}, gives
\begin{eqnarray*}
\delta\left(P_{>0}(\phi,\nu)\right) &=& \lim_{x\to\infty}\sum_{\mycom{\zeta,\text{root of unity}}{ \zeta\in\text{Im}(\chi) ,\re(\zeta^\nu e^{-i\phi})>0}}\dfrac{\pi_{>0}(x,\zeta)}{\pi(x)}+\lim_{x\to\infty}\sum_{\mycom{\zeta,\text{root of unity}}{\zeta\in\text{Im}(\chi) ,\re(\zeta^\nu e^{-i\phi})<0}}\dfrac{\pi_{<0}(x,\zeta)}{\pi(x)}\\
  &=& \frac{1}{2}\sum_{\mycom{\zeta,\text{root of unity}}{ \zeta\in\text{Im}(\chi) ,\re(\zeta^\nu e^{-i\phi})\neq 0}}\frac{1}{r_{\chi}},\\
  &=&\frac{\delta(P_{\neq 0}(\phi,\nu))}{2}.
\end{eqnarray*}
Likewise we get $\delta\left(P_{<0}(\phi,\nu)\right) =\frac{\delta(P_{\neq 0}(\phi,\nu))}{2}$, which concludes the proof.

\Addresses
\end{document}